\documentclass[12pt,reqno]{amsart}
\usepackage{amssymb}
\usepackage{mathrsfs}
\usepackage{cite}
\usepackage{tikz-cd}
\usepackage{MnSymbol}
\usepackage{a4wide}

\newtheorem{theorem}{Theorem}[section]
\newtheorem{lemma}[theorem]{Lemma}
\newtheorem{cor}[theorem]{Corollary}
\newtheorem{prop}[theorem]{Proposition}
\theoremstyle{definition}
\newtheorem{definition}[theorem]{Definition}

\theoremstyle{remark}

\numberwithin{equation}{section}

\newcommand{\dontprint}[1]\relax

\newcommand{\wt}{\widetilde}
\newcommand{\ot}{\otimes}

\newcommand{\si}{\sigma}

\newcommand{\sub}{\subset}

\newcommand{\Gal}{\operatorname{Gal}}

\newcommand{\ov}{\overline}

\newcommand{\la}{\lambda}

\newcommand{\rk}{{\operatorname{rk}}}
\newcommand{\codim}{{\operatorname{codim}}}

\renewcommand{\k}{{\mathbf{k}}}

\newcommand{\srk}{{\operatorname{srk}}}
\newcommand{\cha}{{\operatorname{char}}}

\title{Schmidt rank of quartics over perfect fields}

\author{David Kazhdan}
\author{Alexander Polishchuk}
\thanks{A.P. is partially supported by the NSF grant DMS-2001224, 
and within the framework of the HSE University Basic Research Program and by the Russian Academic Excellence Project `5-100'.}

\address{Einstein Institute of Mathematics,
The Hebrew University of Jerusalem,
Jerusalem 91904, Israel}
\email{kazhdan@math.huji.ac.il}
\address{
    Department of Mathematics, 
    University of Oregon, 
    Eugene, OR 97403, USA; National Research University Higher School of Economics; and Korea Institute for 
    Advanced Study 
  }
  \email{apolish@uoregon.edu}

\begin{document}
\begin{abstract} Let $\k$ be a perfect field of characteristic $\neq 2$.
We prove that the Schmidt rank (also known as strength) of a quartic polynomial $f$ over $\k$ is bounded above in terms of only the Schmidt rank of $f$ over
$\ov{\k}$, an algebraic closure of $\k$. 
\end{abstract}

\maketitle

\section{Introduction}


Recall that the {\it Schmidt rank} (also known as {\it strength}) of a homogeneous polynomial $f\in \k[x_1,\ldots,x_n]$ (see \cite{AKZ}, \cite{BBOV} and references therein)
is  defined as the minimal number $r$ such that $f$ admits a decomposition $f=g_1h_1+\ldots+g_rh_r$, with $\deg(g_i)$ and $\deg(h_i)$ smaller
than $\deg(f)$. We denote the Schmidt rank of $f$ as $\rk^S_{\k}(f)$.

It is conjectured in \cite{AKZ} that for a homogeneous polynomial $f$ of degree $d$ over a non-closed field $\k$ one has 
$$\rk^S_{\k}(f)\le \kappa_d\cdot \rk^S_{\ov{\k}}(f),$$
where $\ov{\k}$ is an algebraic closure of $\k$. This is known to be true for $d\le 3$ with $\kappa_d=d$ since in this case the Schmidt rank is equal to the {\it slice rank} defined as the minimal
$r$ such that $f\in (l_1,\ldots,l_r)$, where $\deg(l_i)=1$ (see \cite[Thm.\ A]{KP-slice} for the case of cubics).
 
One can also ask a weaker question whether there exists a function $c(r,d)$ such that 
$$\rk^S_{\k}(f)\le c(\rk^S_{\ov{\k}}(f),d).$$
Our main result is that this weaker question has a positive answer in the case of quartic polynomials.

\medskip

\noindent
{\bf Theorem A}. {\it Assume that the ground field $\k$ is perfect of characteristic $\neq 2$. Then there exists a function $r\mapsto c(r)$ such that for
any homogeneous quartic polynomial $f(x_1,\ldots,x_n)$, such that $\rk^S_{\ov{\k}}(f)=r$, one has $\rk^S_{\k}(f)\le c(r)$.}

\medskip

We find it convenient to use of the following refined version of Schmidt rank.

\begin{definition} For a collection of nonnegative integers $(r_k,\ldots,r_1)$ and a homogeneous polynomial $f$ of degree $d$, 
we say that the {\it refined Schmidt rank of $f$ is at most $(r_k,\ldots,r_1)$}, and write
$\rk^S_\k(f)\le (r_k,\ldots,r_1)$, if there exists a decomposition over $\k$, 
$$f=\sum_{j=1}^k\sum_{i=1}^{r_j}f_{ij}g_{ij},$$ 
where $\deg(f_{ij})=d-\deg(g_{ij})=j$ 
for $j=1,\ldots,k$.
\end{definition}

\medskip

Thus, for a quartic polynomial $f$, we have $\rk^S_\k(f)\le (r_2,r_1)$ if there exist quadrics $q_1,\ldots,q_{r_2}$ and linear forms $l_1,\ldots,l_{r_1}$, such that
$f\in (q_1,\ldots,q_{r_2},l_1,\ldots,l_{r_1})$.
In our proof of Theorem A we show the existence of functions $(c_2(r_2,r_1),c_1(r_2,r_1))$, such that
if $\rk^S_{\ov{\k}}(f)\le (r_2,r_1)$ then $\rk^S_{\k}(f)\le (c_2(r_2,r_1),c_1(r_2,r_1))$. Then one can set
$$c(r)=\max_{r_1+r_2=r}(c_2(r_2,r_1)+c_1(r_2,r_1)).$$

One can work through our proof of Theorem A and get explicit formulas for $(c_2(r_2,r_1),c_1(r_2,r_1))$.
As an illustration of this, we give formulas for $(c_2(1,r_1), c_1(1,r_1))$.

\medskip

\noindent
{\bf Theorem B}. {\it Let $f$ be a homogeneous quartic polynomial defined over a perfect field $\k$ with $\cha(\k)\neq 2$.
Assume that $\rk^S_{\ov{\k}}(f)\le (1,r)$. Then $\rk^S_{\k}(f)\le (2,C(r))$, where 
$$C(r)=8r(41+20\cdot (10r+1)^{10r+1}).$$
}

\medskip

The main idea of the proof of Theorem A is to study decompositions of a quartic polynomial $f$ of the form
$$f=\sum_{i=1}^r q_iq'_i \mod (P),$$
where $q_i$, $q'_i$ are of degree $2$ and $P$ is a subspace of linear forms.
The main result about such decomposition is that if the rank of any linear combination of $(q_\bullet, q'_\bullet)$ is sufficiently large then
the above decomposition is essentially unique (possibly after enlarging $P$): the only way to get a new decomposition is by making an orthogonal
change of basis in the linear space with the basis $(q_\bullet, q'_\bullet)$.
We then apply this result to the decompositions obtained from a given one over $\ov{\k}$ by applying the Galois group action.
If the rank of $(q_\bullet,q'_\bullet)$ is sufficiently large, then we obtain a $1$-cocycle of the Galois group with values in the orthogonal group measuring how
the decomposition transforms under the Galois action. We can assume that this $1$-cocycle is trivial (after passing to an extension of $\k$ of small degree).
We then use a certain linear algebra result from \cite{KP-lin-alg} to prove the existence of a decomposition with $(q_\bullet,q'_\bullet)$ defined over $\k$.
Furthermore, we have a bound on the slice rank of $f-\sum_{i=1}^r q_iq'_i$ over $\ov{\k}$, and hence over $\k$ by \cite[Thm.\ A]{KP-slice}.
This gives the required bound on the Schmidt rank of $f$ over $\k$.


\section{Preliminaries}

\subsection{Criterion for an ideal generated by quadrics and linear forms to be prime}

In this subsection we fix a ground field $\k$ (and omit it from the notation). By a {\it quadric} we mean an element of $\k[V]_2$, i.e., a quadratic form.

\begin{definition}
For a subspace of quadrics $Q$, and a collection of quadrics $q_1,\ldots,q_r$, we define
$\srk(q_1,\ldots,q_r,Q)$ as the minimum of $\srk(\sum_i c_iq_i+q)$, where $q\in Q$ and $c_i$ are constants, such that either $q\neq 0$ or
$(c_\bullet)\neq 0$. In particular, $\srk(Q)$ is the minimal slice rank of a nonzero element of $Q$.
\end{definition}

We denote by $\rk(q)$ the usual rank of a quadric $q$.
It is easy to see that the rank and the slice rank of a quadric $q$ are related by
$$2\srk(q)-1\le \rk(q)\le 2\srk(q).$$
In other words, we have 
$$\srk(q)=\lceil \frac{\rk(q)}{2} \rceil.$$

\begin{lemma}\label{inter-qu-lem} 
Let $q_1,\ldots,q_r$ be quadratic forms such that 
$$R=\min_{(c_1,\ldots,c_r)\neq 0} \rk(c_1\nu_1+\ldots c_r\nu_r)\ge 2r+1.$$
Then the subscheme $q_1=\ldots=q_r=0$ is normal of codimension $r$.
\end{lemma}

\begin{proof} Consider the Jacobian matrix $J(q_1,\ldots,q_r)$.
The locus $S(q_1,\ldots,q_r)$ where $J(q_1,\ldots,q_r)$ has rank $<r$ coincides with
the union of kernels of $\sum c_j\nu_j$ over $(c_1,\ldots,c_r)\neq 0$. Hence,
$$\codim_V S(q_1,\ldots,q_r)\ge R-(r-1)\ge r+2.$$ 
But this implies that $(q_1=\ldots=q_r=0)$ is a complete intersection, nonsingular in codimension $1$, hence normal.
\end{proof}

\begin{prop}
(i) Let $Q$ be a subspace of quadratic forms such that $\srk(Q)\ge \dim Q+1$. Then the ideal $(Q)$ is prime.

\noindent
(ii) Let $L$ be a subspace of linear forms, $Q$ a subspace of quadratic forms. Assume that $\srk(Q)\ge \dim Q+\dim L+1$.
Then the ideal $(Q,L)$ is prime.
\end{prop}

\begin{proof}
(i) For any $q\in Q$ we have $\rk(q)\ge 2\srk(q)-1\ge 2\dim Q+1$. Hence, by Lemma \ref{inter-qu-lem}, the ideal $(Q)$ is prime.

\noindent
(ii) Consider the quotient $\ov{S}=S/(L)$ of the algebra of polynomials $S$ by the ideal $(L)$. For any $q\in Q$ we have 
$\srk(\ov{q})\ge \srk(q)-\dim L$, where $\ov{q}$ is the image of $q$ in $\ov{S}$. Hence, the image $\ov{Q}$ of $Q$ in $\ov{S}$ satisfies
the assumptions of (i), so the ideal $(\ov{Q})$ in $\ov{S}$ is prime. Therefore, its preimage in $S$, namely $(Q,L)$, is also prime.
\end{proof}

\subsection{Almost invariant quadratic forms}

We will use the following result from \cite{KP-lin-alg}.

\begin{theorem}\label{invariant-operator-thm}
Let $E/\k$ be a finite Galois extension with the Galois group $G$, and let $V_0, V'_0$ be finite dimensional $\k$-vector spaces.
Let us set $V=V_0\ot_{\k} E$, $V'=V'_0\ot_{\k} E$.  
Suppose $T:V\to V'$ is an $E$-linear operator such that for any $\si\in G$, one has 
$$\rk_E(\si(T)-T)\le r,$$ 
for some $r\ge 0$.
Then there exists a $\k$-linear operator $T_0:V_0\to V'_0$, such that 
$$\rk_E(T-(T_0)_E)\le r(2+(r+1)^{r+1}),$$
where $(T_0)_E:V\to V'$ is obtained from $T_0$ by the extension of scalars.
\end{theorem}

We need the following consequence of this theorem for quadratic forms.

\begin{cor}\label{qu-form-cor}
Let $E/\k$ be a finite Galois extension with the Galois group $G$, where $\cha(2)\neq 2$,
and let $V_0$ be a finite dimensional $\k$-vector space, $V=V_0\ot_{\k} E$.
Assume that $q$ is a quadratic form on $V$ such that for any $\si\in G$, one has $\rk_E(\si(q)-q)\le r$ for some $r\ge 0$, where $\rk_E$ is the usual rank of the quadratic form.
Then there exists a quadratic form $q_0$ on $V_0$ such that
$\rk_E(q-q_0)\le 2r(2+(r+1)^{r+1})$.
\end{cor}

\begin{proof}
Let $T:V\to V^*$ be the symmetric linear map associated with $q$. Our assumption implies that $\rk_E(\si(T)-T)\le r$ for any $\si\in G$. By Theorem \ref{invariant-operator-thm},
there exists an operator $T_0:V_0\to V_0^*$ such that $\rk_E(T-T_0)\le r(2+ (r+1)^{r+1})$.
Let $T_0^*:V_0\to V_0^*$ be the dual operator. Then 
$$\rk_E(T-\frac{1}{2}(T_0+T_0^*))\le 2r(2+(r+1)^{r+1}),$$
so we can let $q_0$ be the quadratic form corresponding to $\frac{1}{2}(T_0+T_0^*)$.
\end{proof}

\section{Schmidt rank for quartics}

From now on we assume that the ground field $\k$ is perfect and has characteristic $\neq 2$.

\subsection{Case $r_2=1$}

We start with a proof of Theorem B dealing with the case $r_2=1$, since it is simpler but still shows the main idea.

\begin{lemma}\label{qu-ext-lem} 
Let $\k'/\k$ be a quadratic extension, and let $f$ be a homogeneous polynomial over $\k$ such that $\rk^S_{\k'}(f)\le (r_2,r_1)$. Then $\rk^S_{\k}(f)\le (2r_2,2r_1)$.
\end{lemma}

\begin{proof}
By assumption $f\in (Q,P)$, where $Q$ is a subspace of quadrics and $P$ is a subspace of linear forms, both defined over $\k'$. Hence,
$f\in (Q+\si(Q),P+\si(P))$, where $\si$ is the generator of the Galois group of $\k'/\k$. Since the subspaces $Q$ and $P$ are defined over $\k$,
this implies the assertion.
\end{proof}

\begin{proof}[Proof of Theorem B]

We have to check that if $E/\k$ is a finite Galois extension, and 
$$f\equiv qq' \mod (P),$$ where $q,q'$ are quadratic forms over $E$ and 
$P$ is an $r$-dimensional subspace of linear forms defined over $E$, then 
$\rk^S_{\k}(f)\le (2,C(r))$.

If $\srk_{E}(q)\le 9r$ or $\srk_E(q)\le 9r$ then $\srk_E(f)\le 10r$, and so $\srk_\k(f)\le 40r\le C(r)$.
Thus, we can assume that $\srk_{E}(q)>9r$ and $\srk_E(q')>9r$. Let $G$ be the Galois group of $E/\k$. For any $\si\in G$ we have
$$f\equiv \si(q)\si(q') \mod (\si(P)).$$
By assumption, the slice rank of $\ov{q}=q\mod (P+\si(P))$ is $\ge 2$, hence the quadric $\ov{q}$ is irreducible. In other words, the ideal $(q,P+\si(P))$ is prime.
Since $qq'\in (\si(q),P+\si(P))$, we have either $q\in (\si(q),P+\si(P))$ or $q'\in (\si(q),P+\si(P))$. Since the slice ranks of $q$ and $q'$ are $>2r$, this means that either
$$\si(q)\equiv c(\si)\cdot q \mod (P+\si(P)), \ \si(q')\equiv c(\si)^{-1}\cdot q' \mod (P+\si(P)), \text{ or }$$
$$\si(q)\equiv c(\si)\cdot q' \mod (P+\si(P)), \ \si(q')\equiv c(\si)^{-1}\cdot q \mod (P+\si(P)),$$
for some $c(\si)\in E^*$. 

Let $H\sub G$ be the set of $\si\in G$ for which the first possibility holds. 
Let us consider separately two cases.

\medskip

\noindent
{\bf Case $\srk(q,q')>3r$}. Assume first that $\si_1,\si_2\in H$. Then we have
$$\si_1\si_2(q)\equiv \si_1(c(\si_2))\cdot \si_1(q) \equiv c(\si_1)\si_1(c(\si_2))\cdot q \mod (P+\si_1(P)+\si_1\si_2(P)).$$
If $\si_1\si_2\not\in H$, we would get that a nontrivial linear combination of $q$ and $q'$ is in $P+\si_1(P)+\si_1\si_2(P)$, contradicting the assumption $\srk(q,q')>3r$.
Hence, $\si_1\si_2\in H$. Furthermore, since a nonzero multiple of $q$ cannot be contained in $P+\si_1(P)+\si_1\si_2(P)$, we have
$$c(\si_1)\si_1(c(\si_2))=c(\si_1\si_2),$$
i.e., $c(\si)$ is a $1$-cocycle of $H$. A similar argument shows that if exactly one of $\si_1,\si_2$ belongs to $H$ then $\si_1\si_2\not\in H$, and if $\si_1,\si_2\not\in H$ then $\si_1\si_2\in H$,
proving that $H$ is a subgroup of index $\le 2$ in $G$.

\medskip

\noindent
{\bf Case $\srk(q,q')\le 3r$}. In this case we have $q'\equiv cq \mod(P_0)$ for some subspace of linear forms $P_0$ of dimension $\le 3r$ (defined over $E$) and some $c\in E^*$.
This implies that for $\si\not\in H$ we have
$$\si(q)\equiv c(\si)\cdot q'\equiv c(\si)c\cdot q \mod (P+\si(P)+P_0).$$
Redefining $c(\si)$ for $\si\not\in H$ we obtain that
\begin{equation}\label{P-P0-cocycle-congruence}
\si(q)\equiv c(\si)\cdot q \mod(P+\si(P)+P_0), \ \ \si(q)\equiv c(\si)\cdot q \mod(P+\si(P)+P_0)
\end{equation}
for all $\si\in G$. Now for $\si_1,\si_2\in H$ we have
$$\si_1\si_2(q)\equiv \si_1(c(\si_2))\cdot \si_1(q) \equiv c(\si_1)\si_1(c(\si_2))\cdot q \mod (P+\si_1(P)+\si_1\si_2(P)+P_0+\si_1(P_0)).$$
On the other hand,
$$\si_1\si_2(q)\equiv c(\si_1\si_2)\cdot q\mod(P+\si_1\si_2(P)+P_0).$$
Since $\srk(q)>9$ this implies that $c(\si)$ is a $1$-cocycle of $H$.

In either case we obtain that for a subgroup $H\sub G$ of index $\le 2$ and a subspace of linear forms $P_0$ of dimension $\le 3r$, 
the congruences \eqref{P-P0-cocycle-congruence} hold for some $1$-cocycle $c:H\to E^*$.

Let $\k'/\k$ be the subextension of $E$ corresponding to the subgroup $H\sub G$, so that the extension $E/\k'$ is Galois with the Galois group $H$.
By Hilbert's Theorem 90, the cocycle $c(\si)$ of $H$ is trivial, so rescaling $q$ and $q'$, we can assume that
$$\si(q)\equiv q \mod (P+\si(P)+P_0), \ \ \si(q')\equiv q' \mod (P+\si(P)+P_0)$$
for all $\si\in H$. But this implies that $\srk_E(\si(q)-q)\le 5r$ and $\srk_E(\si(q')-q')\le 5r$.
Hence, we obtain 
$$\rk_E(\si(q)-q)\le 10r, \ \ \rk_E(\si(q')-q')\le 10r$$
for all $\si\in H$.
By Corollary \ref{qu-form-cor}, there exist quadrics $q_0$ and $q'_0$ defined over $\k'$ such that
$$\max(\rk_E(q-q_0),\rk_E(q'-q'_0))\le 20r(2+(10r+1)^{10r+1}).$$
Hence,
$$\max(\srk_E(q-q_0),\srk_E(q'-q'_0))\le 10r(2+(10r+1)^{10r+1}).$$
It follows that
$$\srk_E(f-q_0q'_0)\le \srk_E(f-qq')+\srk_E(qq'-q_0q'_0)\le r+20r(2+(10r+1)^{10r+1})=r(41+20\cdot (10r+1)^{10r+1}).$$
By \cite[Thm.\ A]{KP-slice}, this implies that
$$\srk_{\k'}(f-q_0q'_0)\le 4r(41+20\cdot (10r+1)^{10r+1}),$$
hence, $\rk^S_{\k'}(f)\le (1,4r(41+20\cdot (10r+1)^{10r+1}))$. Since the extension $\k'/\k$ is either trivial or quadratic, applying Lemma \ref{qu-ext-lem} we get the result.
\end{proof}

\subsection{Quadratic decompositions of quartics}\label{quadr-decomp-sec}

\begin{lemma}\label{qu-decomp-lem}
Let $Q$ be a subspace of quadrics, 
$$q_1,\ldots,q_r,q'_1,\ldots,q'_r,p_1,\ldots,p_s,p'_1,\ldots,p'_s$$ 
quadratic forms, where $r>s$, such that
$$\sum_{i=1}^r q_iq'_i\equiv \sum_{i=1}^s p_ip'_i \mod (Q).$$
Then for some constants $a_1,\ldots,a_r,a'_1,\ldots,a'_r$ such that $\sum_i a_ia'_i=0$, we have
$$\srk(\sum_i (a_iq_i+a'_iq'_i),Q)\le c(r,s,\dim Q):=2^s(r+\dim Q)+2^{s-1}(s-2).$$
\end{lemma}

\begin{proof}
We use the induction on $s$. In the case $s=0$ we have to prove that
$$\srk(\sum_i (a_iq_i+a'_iq'_i),Q)\le c(r,0,\dim Q)=r+\dim Q-1$$ 
for some isotropic $(a_\bullet,a'_\bullet)$.
Indeed, assume this is not true. Then $\srk(q_1,\ldots,q_{r-1},Q)\ge r+\dim Q$, so
the ideal $(q_1,\ldots,q_{r-1},Q)$ is prime. But we have
$$q_rq'_r\in (q_1,\ldots,q_{r-1},Q).$$
Hence, swapping $q_r$ with $q'_r$ if necessary, we
deduce that $q_r\in (q_1,\ldots,q_{r-1},Q)$ which is impossible since $\srk(q_\bullet,Q)\ge r+\dim Q\ge r>0$.

Assume the assertion holds for $s-1$. We have
$$\sum_{i=1}^r q_iq'_i\equiv \sum_{i=2}^s p_ip'_i \mod (p_1,Q).$$
Hence, by the induction assumption, 
$$\srk(\sum_i (a_iq_i+a'_iq'_i),p_1,Q)\le c(r,s-1,\dim Q+1)$$
for some isotropic $(a_\bullet,a'_\bullet)$.
Changing the basis in $(q_\bullet)$, we can assume that either $\srk(q_1,Q)\le c(r,s-1,\dim Q+1)$, or
there exists a subspace $L$ of linear forms of dimension $\le c(r,s-1,\dim Q+1)$ such that
$p_1\in (q_1,Q,L)$. In the former case we are done since $c(r,s,\dim Q)\ge c(r,s-1,\dim Q+1)$.
In the latter case, we have
$$\sum_{i=2}^r q_iq'_i\equiv \sum_{i=2}^s p_ip'_i \mod (q_1,Q,L).$$
Applying the induction assumption we obtain that there exists an isotropic vector $(a_{>1},a'_{>1})$ such that
$$\srk_L(\sum_{i=2}^r (a_iq_i+q'_iq'_i),q_1,Q)\le c(r-1,s-1,\dim Q+1),$$
hence 
\begin{align*}
&\srk(\sum_{i=2}^r (a_iq_i+q'_iq'_i),q_1,Q)\le c(r-1,s-1,\dim Q+1)+\dim L\le \\
&c(r-1,s-1,\dim Q+1)+c(r,s-1,\dim Q+1)=c(r,s,\dim Q).
\end{align*}
Since any linear combination of $\sum_{i=2}^r (a_iq_i+q'_iq'_i)$ with $q_1$ will correspond to an isotropic vector, the assertion follows.
\end{proof}

\begin{prop}\label{qu-decomp-prop}
Let $Q$ be a subspace of quadrics, 
$$q_1,\ldots,q_r,q'_1,\ldots,q'_r,p_1,\ldots,p_r,p'_1,\ldots,p'_r$$ quadratic forms, such that
\begin{equation}\label{2-quad-decomp-Q-eq}
\sum_{i=1}^r q_iq'_i\equiv \sum_{i=1}^r p_ip'_i \mod (Q).
\end{equation}
Assume that for any constants $a_1,\ldots,a_r,a'_1,\ldots,a'_r$ such that $\sum_i a_ia'_i=0$, we have
$$\srk(\sum_i (a_iq_i+a'_iq'_i),Q)\ge C(r,\dim Q):=2^r(r+\dim Q)+2^{r-1}(r-2)+1.$$
Then there exists a subspace of linear forms $L$ 
of dimension at most
$$D(r,\dim Q):=(2^r-1)(r+\dim Q-1)+r\cdot 2^{r-1}$$
and a linear transformation $A:\k^{2r}\to \k^{2r}$ preserving the quadratic form $\sum_{i=1}^r x_iy_i$,
such that for the linear operator $\phi$ from $\k^{2r}$ to the space of quadrics sending the standard basis $(e_\bullet,f_\bullet)$ to $(q_\bullet,q'_\bullet)$, we have
$$p_i\equiv \phi(Ae_i), \ \ p'_i\equiv \phi(Af_i) \mod (Q,L).$$
\end{prop}

\begin{proof}
We use induction on $r$. In the case $r=0$ we can take $L=0$, $D(0,\dim Q)=0$.

Assume the assertion holds for $r-1$. We have 
$$\sum_{i=1}^r q_iq'_i\equiv \sum_{i=2}^r p_ip'_i \mod (p_1,Q).$$
Hence, by Lemma \ref{qu-decomp-lem}, 
changing $(q_\bullet,q'_\bullet)$ by an orthogonal transformation, we can achieve that
$$\srk(q_1,p_1,Q)\le c(r,r-1,\dim Q+1).$$
Since $\srk(q_1,Q)\ge C(r,\dim Q)\ge c(r,r-1,\dim Q+1)+1$, this implies that there exists a subspace of linear forms $L$ of
dimension $\le c(r,r-1,\dim Q+1)$, such that
$$p_1\in (q_1,Q,L).$$

Note that if $p_1\in (Q,L)$ then we get
$$\sum_{i=1}^r q_iq'_i\equiv \sum_{i=2}^r p_ip'_i \mod (Q,L).$$
Hence, by Lemma \ref{qu-decomp-lem}, 
we would get
\begin{align*}
&\srk(\sum_i (a_iq_i+a'_iq'_i),Q)\le \srk_L(\sum_i (a_iq_i+a'_iq'_i),Q)+\dim L\le \\
&c(r,r-1,\dim Q)+c(r,r-1,\dim Q+1)\le C(r,\dim Q)-1,
\end{align*}
which is a contradiction. Hence, rescaling $q_1$ and $q'_1$, we can assume that
$$p_1\equiv q_1 \mod (Q,L).$$ 

Also, from $p_1\in (q_1,Q,L)$ we deduce that 
$$\sum_{i=2}^r q_iq'_i\equiv \sum_{i=2}^r p_ip'_i \mod (q_1,Q,L).$$
Since for any isotropic vector $(a_{>1},a'_{>1})$ one has
\begin{align*}
&\srk_L(\sum_{i=2}^r (a_iq_i+a'_iq'_i),q_1,Q)\ge \srk(\sum_{i=2}^r (a_iq_i+a'_iq'_i),q_1,Q)-\dim L\ge C(r,\dim Q)-\dim L\ge \\
&C(r,\dim Q)-c(r,r-1,\dim Q+1)\ge C(r-1,\dim Q+1),
\end{align*}
we can apply the induction hypothesis and deduce that for some subspace of linear forms $L'\supset L$ of dimension
$$D(r-1,\dim Q+1)+\dim L\le D(r-1,\dim Q+1)+c(r,r-1,\dim Q+1)=D(r,\dim Q),$$ 
after changing the basis $(q_2,\ldots,q_r,q'_2,\ldots,q'_r)$ by an orthogonal transformation, we have
$$p_i\equiv q_i, \ \ p'_i\equiv q'_i \mod (q_1,Q,L'),$$
for $i\ge 2$.
This means that we have
$$p_i\equiv q_i+c_iq_1, \ \ p'_i\equiv q'_i+c'_iq_1 \mod (Q,L'),$$
for $i\ge 2$.
Substituting this into \eqref{2-quad-decomp-Q-eq} and recalling that $p_1\equiv q_1 \mod (Q,L)$, we get
$$q_1q'_1\equiv q_1\cdot [p'_1+\sum_{i=2}^r (c_iq'_i+c'_iq_i)+(\sum_{i=2}^r c_ic'_i)q_1] \mod (Q,L').$$
Since 
$$\srk(Q)\ge C(r,\dim Q)\ge \dim Q+D(r,\dim Q)+1\ge \dim Q+\dim L'+1,$$
the ideal $(Q,L')$ is prime, so we get
$$p'_1\equiv q'_1-\sum_{i=2}^r (c_iq'_i+c'_iq_i)-(\sum_{i=2}^r c_ic'_i)q_1 \mod (Q,L').$$

It remains to observe that the linear transformation
\begin{align*}
&Ae_1=e_1, \ \ Af_1=f_1-\sum_{i=2}^r (c_if_i+c'_ie_i)-(\sum_{i=2}^r c_ic'_i)e_1, \\ 
&Ae_i=e_i+c_ie_1, \ \ Af_i=f_i+c'_ie_1, \text{ for } i\ge 2,
\end{align*}
preserves the quadratic form $\sum x_iy_i$.
\end{proof}

We are mainly interested in the case $Q=0$ in the above proposition (the case of general $Q$ was introduced in order for the inductive argument to work).

\begin{cor}\label{quad-decomp-cor}
(i) Assume that 
$$\sum_{i=1}^r q_iq'_i=\sum_{i=1}^r p_ip'_i,$$
where $q_i$, $q'_i$, $p_i$, $p'_i$ are quadrics and 
$$\srk(\sum_i (a_iq_i+a'_iq'_i))\ge C(r,0)=(r-1)\cdot 2^r+r\cdot 2^{r-1}+1$$
for any isotropic $(a_\bullet,a'_\bullet)$.
Then there exists a subspace of linear forms $L$ 
of dimension at most
$$D(r,0)=(r-1)\cdot (2^r-1)+r\cdot 2^{r-1}\le C(r,0)$$
and a linear transformation $A:\k^{2r}\to \k^{2r}$ preserving the quadratic form $\sum_{i=1}^r x_iy_i$,
such that for the linear operator $\phi$ from $\k^{2r}$ to the space of quadrics sending the standard basis $(e_\bullet,f_\bullet)$ to $(q_\bullet,q'_\bullet)$, we have
$$p_i\equiv \phi(Ae_i), \ \ p'_i\equiv \phi(Af_i) \mod (L).$$

\noindent
(ii) Assume that 
$$q_0^2+\sum_{i=1}^r q_iq'_i=p_0^2+\sum_{i=1}^r p_ip'_i,$$
where $q_i$, $q'_i$, $p_i$, $p'_i$ are quadrics and 
for any constants $a_0,\ldots,a_r,a'_1,\ldots,a'_r$ such that
$a_0^2+4\sum_i a_ia'_i=0$, one has
$\srk(a_0q_0+\sum_{i=1}^r (a_iq_i+a'_iq'_i)\ge c(r,r-1,0)+C(r,0)+1$.
Then there exists a subspace of linear forms $L$ of dimension at most $D(r,0)+c(r+1,r,0)$, such that after
making a linear change in $(q_0,\ldots,q_r,q'_1,\ldots,q'_r)$ preserving the quadratic form $a_0^2+4\sum_{i=1}^r a_ia'_i$,
one has
$$q_0\equiv p_0, \ \ q_i\equiv p_i,  \ \  q'_i\equiv p'_i \mod (L).$$
\end{cor}

\begin{proof}
(i) This is the case $Q=0$ of Proposition \ref{qu-decomp-prop}.

\noindent
(ii) We have
$$(q_0-p_0)(q_0+p_0)+\sum_{i=1}^r q_iq'_i\equiv \sum_{i=1}^r p_ip'_i.$$
Hence, by Lemma \ref{qu-decomp-lem}, there exists a nonzero vector $(a_0,\ldots,a_r,a'_1,\ldots,a'_r,b)$ such that
$$(a_0-b)(a_0+b)+\sum_{i=1}^r a_ia'_i=0$$ 
and 
$$\srk((a_0-b)(q_0+p_0)+(a_0+b)(q_0-p_0)+\sum_{i=1}^r(a_iq_i+a'_iq'_i))\le c(r+1,r,0).$$
We can rewrite these conditions as
$$a_0^2+\sum_{i=1}^r a_ia'_i=b^2$$
and $\srk(a_0q_0+\frac{1}{2}\sum_{i=1}^r(a_iq_i+a'_iq'_i)-bp_0,Q)\le c(r+1,r,0)$.
Note that by the assumption, we necessarily have $b\neq 0$.
Thus, after making an orthogonal transformation of $(q_0,\ldots,q_r,q'_1,\ldots,q'_r)$, we can assume
that $\srk (q_0-p_0)\le c(r+1,r,0)$. Thus, we have $q_0\equiv p_0 \mod (L_0)$,
$$\sum_{i=1}^r q_iq'_i\equiv \sum_{i=1}^r p_ip'_i \mod (L_0),$$
for some subspace of linear forms $L_0$ of dimension $\le c(r+1,r,0)$.

Now applying part (i), we find a subspace of linear forms $L\supset L_0$ of dimension $\le D(r,0)+c(r+1,r,0)$, such that
after an orthogonal change of $(q_1,\ldots,q_r,q'_1,\ldots,q'_r)$, we have
$$p_i\equiv q_i \mod (L),  \ \ p'_i\equiv q'_i \mod (L).$$
\end{proof}

\subsection{Proof of Theorem A}


We will prove the existence of functions $c_1(r_2,r_1)$ and $c_2(r_2,r_1)$ such that
if a quartic $f$ satisfies $\rk^S_{\ov{\k}}(f)\le (r_2,r_1)$ then $\rk^S_{\k}(f)\le (c_2(r_2,r_1),c_1(r_2,r_1))$.

We use the induction on $r_2$. 
In the case $r_2=0$ we just have the slice rank, so by \cite[Thm.\ A]{KP-slice}, we can set 
$$c_1(0,r_1)=4r_1, \ \ c_2(0,r_1)=0.$$

Now assume that the functions $c_1(r_1,r_2)$ and $c_2(r_1,r_2)$ are already constructed for $r_2<r$.
Let $f$ be a quartic over $\k$, and $E/\k$ is a finite Galois extension such that $\rk^S_{E}(f)\le (r,p)$, i.e.,
over $E$ we have a decomposition
\begin{equation}\label{f-qq'-P-decomp}
f\equiv \sum_{i=1}^r q_iq'_i \mod (P),
\end{equation}
where $(q_i,q'_i)$ are quadrics, and $P$ is a subspace of linear forms of dimension $p$.

Below we will use constants $C(r,0)$, $D(r,0)$ and $c(r,s,0)$ introduced in Sec.\ \ref{quadr-decomp-sec}.

\medskip

\noindent
{\bf Case 1}. Assume first that $\srk(q_\bullet,q'_\bullet)>6p+3C(r,0)$.

Note that for any element $\si$ of the Galois group $\Gal(E/\k)$ we have
$$\sum_{i=1}^r q_iq'_i \equiv \sum_{i=1}^r \si(q_i)\si(q'_i) \mod (P+\si(P)).$$ 
Since $\srk_{P+\si(P)}(\sum_i (a_iq_i+a'_iq'_i))>4p+3C(r,0)\ge C(r,0)$
for any nonzero $(a_\bullet,a'_\bullet)$, by Corollary \ref{quad-decomp-cor}(i),
there exists a subspace of linear forms $L_\si\supset P+\si(P)$
of dimension $\le N=2p+C(r,0)$ and an orthogonal transformation $A_\si$ of the $2r$ dimensional space with the basis $(q_\bullet,q'_\bullet)$ such that
$$\si(q_i)\equiv A_\si(q_i), \ \ \si(q'_i)\equiv A_\si(q'_i) \mod (L_\si)$$
(note that here $A_\si(q_i)$ and $A_\si(q'_i)$ are linear combinations of $(q_\bullet,q'_\bullet)$). 

We claim that $\si\mapsto A_\si$ defines a cocycle with values in the orthogonal group.
Indeed, we have
$$\si_1(A_{\si_2})A_{\si_1}(q_i)\equiv \si_1(A_{\si_2})(\si_1(q_i))\equiv \si_1\si_2(q_i) \mod(\si_1(L_{\si_2})+L_{\si_1}),$$
hence,
$$q_i\equiv A_{\si_1\si_2}^{-1}\si_1(A_{\si_2})A_{\si_1}(q_i) \mod(\si_1(L_{\si_2})+L_{\si_1}+L_{\si_1\si_2}).$$
Similarly,
$$q'_i\equiv  A_{\si_1\si_2}^{-1}\si_1(A_{\si_2})A_{\si_1}(q'_i) \mod(\si_1(L_{\si_2})+L_{\si_1}+L_{\si_1\si_2}).$$
Since $\srk(q_\bullet,q'_\bullet)>3N\ge \dim(\si_1(L_{\si_2})+L_{\si_1}+L_{\si_1\si_2})$,  this implies 
that $A_{\si_1\si_2}=\si_1(A_{\si_2})A_{\si_1}$.

Recall that $H^1$ of the Galois group with values in the orthogonal group classifies nondegenerate quadratic forms on $\k^{2r}$,
and any such quadratic form becomes equivalent to $\sum_i x_iy_i$ over an extension
$\k'$ obtained by adjoining at most $r$ square roots to $\k$. Indeed, this follows from the identity
$$\la_1 x^2+\la_2 y^2=\la_1(x+\sqrt{-\la_2/\la_1}y)(x-\sqrt{-\la_2/\la_1}y).$$

In our case we have the class in $H^1$ of the Galois group of $E/\k$, hence we can choose $\k'$ to be a subextension of $E$,
and the cocycle $\si\mapsto A_\si$ becomes a coboundary when restricted to $\Gal(E/\k')$.
Thus, we can make an orthogonal change of basis in $(q_\bullet,q'_\bullet)$ such that
$$\si(q_i)\equiv q_i, \ \ \si(q'_i)\equiv q'_i \mod (L_\si)$$
for any $\si$ in the Galois group $\Gal(E/\k')$.
By Corollary \ref{qu-form-cor}, there exist quadratic forms $\ov{q}_\bullet$, $\ov{q}'_\bullet$ defined over $\k'$, such that
$$\srk(q_i-\ov{q}_i)\le N', \ \ \srk(q'_i-\ov{q}'_i)\le N',$$
with $N'=N(2+(2N+1)^{2N+1})$.

Thus, we have
$$f-\sum_{i=1}^r \ov{q}_i\ov{q}'_i \in (P')$$
for some subspace of linear forms $P'$ over $E$ of dimension $\le N''=p+2rN'$.
In other words, the slice rank of $\wt{f}=f-\sum_{i=1}^r \ov{q}_i\ov{q}'_i$ over $E$ is $\le N''$.
By \cite[Thm.\ A]{KP-slice}, this implies that the slice rank of $\wt{f}$ over $\k'$ is $\le 4N''$.
This means that
$$\rk^S_{\k'}(f)\le (r,4N'').$$
By Lemma \ref{qu-ext-lem}, it follows that
$$\rk^S_{\k'}(f)\le (2^r\cdot r,2^r\cdot 4N'').$$

\medskip

\noindent
{\bf Case 2}. Next, assume that $\srk(q_\bullet,q'_\bullet)\le 3N$, so there exists a nontrivial linear combination
$\sum_i a_i q_i+\sum_i a'_iq'_i$ which has slice rank $\le 3N$. The restriction of the quadratic form $\sum_{i=1}^r x_iy_i$ to
the linear subspace $\sum_i a_ix_i+ \sum_i a'_ix'_i=0$ has rank $2r-1$ or $2r-2$. So enlarging $E$ if necessary we can find
linear combinations $\ov{q}_0,(\ov{q}_i,\ov{q}'_i)_{i=1}^{r-1}$ of $(q_\bullet,q'_\bullet)$ (where possibly $\ov{q}=0$) such that
$$\sum_{i=1}^r q_iq'_i\equiv \ov{q}_0^2+\sum_{i=1}^{r-1} \ov{q}_i\ov{q}'_i \mod (\sum_i a_i q_i+\sum_i a'_iq'_i).$$
Hence, renaming $\ov{q}_i$, $\ov{q}'_i$ by $q_i$, $q'_i$, we obtain
$$f\equiv q_0^2+\sum_{i=1}^{r-1} q_iq'_i \mod (P'),$$
where $P$ is a space of linear forms of dimension $\le p'=\dim P\le p+3N$.

Furthermore, we claim that we can assume that $\srk(q_0,q_\bullet,q'_\bullet)>3N'$, where $N'=D(r-1,0)+c(r,r-1,0)$.
Indeed, otherwise arguing as above we obtain a decomposition \eqref{f-qq'-P-decomp} with $r$ replaced by $r-1$ and $p$ replaced by $p''=p'+3N'$.
In other words, we would have $\rk^S_{E}(f)\le (r-1,p'')$, so by the induction assumption we would deduce that
$$\rk^S_{\k}(f)\le (c_2(r-1,p''),c_1(r-1,p'')).$$

Since for every $\si\in\Gal(E/\k)$, one has
$$q_0^2+\sum_{i=1}^{r-1} q_iq'_i\equiv \si(q_0)^2+\sum_{i=1}^{r-1} \si(q_i)\si(q'_i) \mod (P'+\si P'),$$
by Corollary \ref{quad-decomp-cor}(ii), we get that 
$$\si(q_0)\equiv A_\si q_0, \ \ \si(q_i)\equiv A_\si q_i, \ \ \si(q'_i)\equiv A_\si q'_i \mod (L_\si),$$
for some orthogonal transformation $A_\si$ and a
subspace of linear forms $L_\si\supset P'+\si P'$ of dimension $\le N'$
(here we use the inequality 
$$3D(r-1,0)+3c(r,r-1,0)\ge C(r-1,0)+c(r-1,r-2,0)$$
 which is easy to check).
Since $\srk(q_0,q_\bullet,q'_\bullet)> 3N'$, as in Case 1, this implies that $\si\mapsto A_\si$ is a $1$-cocycle.

Arguing as in Case 1, we find a subextension $\k'\sub E$ obtained by adjoing at most $r-1$ square roots to $\k$, such that
after making a change of basis in $(q_\bullet, q'_\bullet)$, we get
$$\si(q_i)\equiv q_i, \ \si(q'_i)\equiv q'_i \mod (L_\si)$$
for any $\si\in \Gal(E/\k')$. Hence, by Corollary \ref{qu-form-cor}, there exist quadratic forms $\ov{q}_\bullet$, $\ov{q}'_\bullet$ defined over $\k'$, such that
$$\srk(q_i-\ov{q}_i)\le N'' \text{ for } i=0,\ldots,r-1, \ \ \srk(q'_i-\ov{q}'_i)\le N'' \text{ for } i=1,\ldots,r-1,$$
with $N''=N'(2+(2N'+1)^{2N'+1})$, and we get
$$\srk_E(f-\ov{q}_0^2-\sum_{i=1}^{r-1}\ov{q}_i\ov{q}'_i)\le M=p'+(2r-1)N''.$$
By \cite[Thm.\ A]{KP-slice}, this implies that the slice rank of $\wt{f}$ over $\k'$ is $\le 4M$, so we get
$$\rk^S_{\k'}(f)\le (r,4M),$$
and so by Lemma \ref{qu-ext-lem}, 
$$\rk^S_{\k'}(f)\le (2^r\cdot r,2^r\cdot 4M).$$

\end{document}